\documentclass[a4paper, 10pt]{amsart}

\usepackage{arydshln}
\usepackage[latin1]{inputenc}
\usepackage[shortlabels]{enumitem}
\usepackage{nameref}
\usepackage{hyperref}
\usepackage[margin=3cm]{geometry}
\usepackage[nospace,noadjust]{cite}
 \usepackage[T1]{fontenc}
\usepackage{lastpage, setspace}
\usepackage{pst-all}
\usepackage[all]{xy}
\usepackage{amsmath, amsthm, amssymb, tensor}
\usepackage{listings}
\usepackage{verbatim}
\usepackage{amsmath,bm}

\usepackage{pdflscape}

\makeatletter
\newcommand{\displaybump}{\hbox to \@totalleftmargin{\hfil}}
\makeatother
\makeatletter
\def\namedlabel#1#2{\begingroup
    #2%
    \def\@currentlabel{#2}%
    \phantomsection\label{#1}\endgroup
}
\makeatother

\SelectTips{cm}{}
\newdir{ >}{{}*!/-5pt/\dir{>}}
\DeclareMathOperator\Aut{Aut}

\DeclareMathOperator\GL{GL}
\DeclareMathOperator\SL{SL}
\DeclareMathOperator\SO{SO}

\DeclareMathOperator\id{id}
\DeclareMathOperator\tr{tr}

\DeclareMathOperator\stab{stab}
\DeclareMathOperator\supp{supp}

\DeclareMathOperator\bbC{\mathbb{C}}
\DeclareMathOperator\bbH{\mathbb{H}}
\DeclareMathOperator\bbN{\mathbb{N}}
\DeclareMathOperator\bbP{\mathbb{P}}
\DeclareMathOperator\bbQ{\mathbb{Q}}
\DeclareMathOperator\bbR{\mathbb{R}}
\DeclareMathOperator\bbZ{\mathbb{Z}}

\DeclareMathOperator\calB{\mathcal{B}}

\DeclareMathOperator\calM{\mathcal{M}}
\DeclareMathOperator\calN{\mathcal{N}}

\DeclareMathOperator\calP{\mathcal{P}}

\theoremstyle{definition}
\newtheorem{theorem}{Theorem}[section]
\newtheorem*{theorem*}{Theorem}
\newtheorem{lemma}[theorem]{Lemma}
\newtheorem*{lemma*}{Lemma}
\newtheorem{corollary}[theorem]{Corollary}
\newtheorem{definition}[theorem]{Definition}
\newtheorem*{definition*}{Definition}
\newtheorem{remark}[theorem]{Remark}
\newtheorem*{remark*}{Remark}
\newtheorem{proposition}[theorem]{Proposition}
\newtheorem*{proposition*}{Proposition}
\newtheorem{example}[theorem]{Example}
\newtheorem*{example*}{Example}
\newtheorem*{sketch of proof}{Sketch of Proof}
\newtheorem*{idea of proof}{Idea of Proof}

\title{Haar Measures}
\author{Stephan Tornier}
\date{\today}

\begin{document}

\begin{abstract}
This article provides a concise introduction to the theory of Haar measures on locally compact Hausdorff groups. We cover the necessary preliminaries on topological groups and measure theory, the Haar correspondence, unimodularity and Haar measures on coset spaces.
\end{abstract}

\maketitle

\section{Preliminaries}

References appear throughout the article. Apart from the classics by Haar \cite{Haa33}, Weil \cite{Wei65} and Bourbaki~\cite{Bou04}, the neat introduction by Knightly-Li~\cite[Sec. 7]{KL06} deserves highlighting.

\subsection{Locally Compact Hausdorff Groups}
The natural class of groups for which to consider Haar measures is that of locally compact Hausdorff groups which we review presently.

\vspace{0.2cm}
A \emph{topological group} is a group $G$ with a topology such that the multiplication map $G\times G\to G$ and the inversion map $G\to G$ are continuous. As a consequence, left and right multiplication by elements of $G$ as well as inversion are homeomorphisms of $G$. Therefore, the neighbourhood system of the identity $e\in G$ determines the topology on $G$. A topological space is \emph{locally compact} if every point has a compact neighbourhood; and it is \emph{Hausdorff} if any two distinct points have disjoint neighbourhoods. In the Hausdorff case, local compactness is equivalent to every point admitting a \emph{relatively compact} open neighbourhood, i.e. an open neighbourhood with compact closure.

\vspace{0.2cm}
The class of locally compact Hausdorff groups is stable under taking closed subgroups as the following proposition shows. Recall that if $X$ is a topological space and $A$ is a subset of $X$, we may equip $A$ with the \emph{subspace topology}, for which $U\subseteq A$ is open if and only if there is an open set $V\subseteq X$, such that $U=A\cap V$.

\begin{proposition}
Let $X$ be a locally compact Hausdorff space and let $A$ be a closed subset. Then $A$ is locally compact Hausdorff.
\end{proposition}

\begin{proof}
Recalling that compact subsets of Hausdorff spaces are closed and that closed subsets of compact sets are compact, this is immediate following the definitions.
\end{proof}

As to coset spaces, we record the following lemma on a property of neighbourhoods that comes with the group structure.

\begin{lemma}\label{lem:prod_nbhd}
Let $G$ be a topological group. Then for every $x\in G$ and every neighbourhood $U$ of $e\in G$, there exists an open neighbourhood $V$ of $x$ such that $V^{-1}V\subseteq U$. 
\end{lemma}

\begin{proof}
The map $\varphi:G\times G\to G,\ (g,h)\mapsto g^{-1}h$ is continuous. Hence there are open sets $V_{1},V_{2}\subseteq G$ such that $V_{1}^{-1}V_{2}=\varphi(V_{1}\times V_{2})\subseteq U$. Then $V=V_{1}\cap V_{2}$ serves.
\end{proof}

When $G$ is a topological group and $H\le G$ is a subgroup of $G$, we equip the set of cosets $G/H$ with the \emph{quotient topology}, i.e. $U\subseteq G/H$ is open if and only if $\pi^{-1}(U)\subseteq G$ is open, where $\pi:G\to G/H,\ g\mapsto gH$. Then $\pi$ is continuous and open, and left multiplication by $g\in G$ is a homeomorphism of $G/H$.

\begin{proposition}\label{prop:quot_haus}
Let $G$ be a topological group and let $H\le G$ be closed. Then $G/H$ is Hausdorff.
\end{proposition}

\begin{proof}
Let $xH, yH\in G/H$ be distinct. Then $yHx^{-1}\subseteq G$ is closed and does not contain $e\in G$. Hence, by Lemma~\ref{lem:prod_nbhd}, there is an open neighbourhood $V\subseteq G$ of $e\in G$ with $V^{-1}V\subseteq G\backslash yHx^{-1}$. Then $VxH$ and $VyH$ are disjoint neighbourhoods of $xH\in G/H$ and $yH\in G/H$ respectively.
\end{proof}

\begin{proposition}\label{prop:quot_loccpct}
Let $G$ be a locally compact group and let $H\le G$. Then $G/H$ is locally compact.
\end{proposition}

\begin{proof}
It suffices to show that $H\in G/H$ has a compact neighbourhood. Since $G$ is locally compact, there is a compact neighbourhood $U$ of $e\in G$. Let $V$ be as in Lemma~\ref{lem:prod_nbhd}. Then $\pi(V)$ is an open neighbourhood of $H\in G/H$ since $\pi$ is open. We show that $\smash{\overline{\pi(V)}}$ is compact. If $\smash{gH\in\overline{\pi(V)}}$ then $VgH\cap VH\neq\emptyset$ and hence $gH=v_{1}^{-1}v_{2}H$ for some $v_{1},v_{2}\in V$. Thus $\smash{\overline{\pi(V)}}\subseteq\pi(U)$. The latter set is compact since $\pi$ is continuous and hence so is $\smash{\overline{\pi(V)}}\subseteq\pi(U)$.
\end{proof}

We now state a version of Urysohn's Lemma which guarantees the existence of certain compactly supported functions on locally compact Hausdorff spaces. Recall that when $X$ is a topological space, $C_{c}(X)$ denotes the set of continuous, complex-valued functions $f$ on $X$ with compact support $\smash{\mathrm{supp}(f):=\overline{\{x\in X\mid f(x)\neq 0\}}}$. When $f\in C_{c}(X)$ is such that $0\le f(x)\le 1$ for all $x\in X$, $U\subseteq X$ is open and $K\subseteq X$ is compact, write $f\prec U$ if $\supp(f)\subseteq U$ and $K\prec f$ if $f(k)=1$ for all $k\in K$.

\begin{lemma}\label{lem:urysohn_aux}
Let $X$ be a locally compact Hausdorff space. When $K\subseteq X$ is compact and $U\subseteq X$ is open with $K\subseteq U$, there exists an open set $V\subseteq X$ with compact closure such that $K\subseteq V\subseteq\overline{V}\subseteq X$.
\end{lemma}

\begin{proof}
By compactness of $K$ and local compactness of $X$, there is a relatively compact open set $W$ containing $K$. Using once more that $K$ is compact, and that $X$ is Hausdorff, there is for every $p\in U^{c}$ an open set $V_{p}$ containing $K$ such that $p\not\in\overline{V}_{p}$. Then $\smash{(U^{c}\cap\overline{W}\cap\overline{V}_{p})_{p\in C}}$ is a family of compact sets with empty intersection. Hence there are $p_{1},\ldots,p_{n}\in U^{c}$ such that $\smash{\bigcap_{i=1}^{n}U^{c}\cap\overline{W}\cap\overline{V}_{p_{i}}}$ is empty as well. Set $V:=W\cap\bigcap_{i=1}^{n}V_{p}$.
\end{proof}

\begin{lemma}[Urysohn]\label{lem:urysohn}
Let $X$ be a locally compact Hausdorff space. When $K\subseteq X$ is compact and $U\subseteq X$ is open such that $K\subseteq U$, then there exists $f\in C_{c}(G)$ satisfying $K\prec f\prec U$.
\end{lemma}

\begin{proof}
Let $r:\bbN_{0}\to\bbQ\cap[0,1]$ be a bijection with $r(0)=0$ and $r(1)=1$. Using Lemma~\ref{lem:urysohn_aux}, pick open sets $U_{r(1)}$ and $U_{r(0)}$ with compact closure such that $\smash{K\subseteq U_{r(1)}\subseteq\overline{U}_{r(1)}\subseteq U_{r(0)}\subseteq\overline{U}_{r(0)}\subseteq U}$. Then, by induction on $n\in\bbN_{0}$ and using Lemma~\ref{lem:urysohn_aux}, construct open sets $U_{r(n)}$ with compact closure such that for all $s,t\in\bbQ\cap[0,1]$ with $s>r(n)>t$ we have $\smash{\overline{V}_{s}\subseteq V_{r(n)}\subseteq\overline{V}_{r(n)}\subseteq V_{t}}$. Given $\alpha\in[0,1]$, set $U_{\alpha}:=\bigcup_{t\in\bbQ\cap[\alpha,1]}U_{t}$ and define
\begin{displaymath}
 \displaybump f:X\to\bbR,\ x\mapsto\begin{cases}1 & x\in U_{1} \\ \inf\{\alpha\in[0,1]\mid x\in U_{\alpha}\} & x\not\in U_{1}\end{cases}.
\end{displaymath}
For continuity, let $x\in X$ and $0<\delta<\varepsilon$. Then $x\in U_{f(x)-\varepsilon-\delta}\backslash\overline{U}_{f(x)+\varepsilon}\subseteq f^{-1}((f(x)-\varepsilon,f(x)+\varepsilon))$, where $U_{\alpha}:=X$ for $\alpha<0$ and $U_{\alpha}:=\emptyset$ for $\alpha>1$. Overall, $f\in C_{c}(G)$ and $K\prec f\prec U$.
\end{proof}

We also need the notion of uniform continuity for functions on topological groups (which comes from giving the group the structure of a uniform space). Let $G$ be a topological group. A function $f:G\to\bbC$ is \emph{uniformly continuous on the left (right)} if for all $\varepsilon>0$ there is an open neighbourhood $U$ of $e\in G$ such that for all $x\in G$ and $g\in U$ we have $\vert f(gx)-f(x)\vert<\varepsilon$ $(\vert f(xg)-f(x)\vert<\varepsilon)$.

\begin{proposition}\label{prop:unif_cont}
Let $G$ be a locally compact Hausdorff group. Then any $f\in C_{c}(G)$ is uniformly continuous on the left and right.
\end{proposition}

\begin{proof}
We prove that $f$ is uniformly continuous on the left. Uniform continuity on the right can be handled analogously. Let $\varepsilon>0$. By continuity of $f$, there is for each $x\in\supp f$ an open neighbourhood $U_{x}$ of $e\in G$ such that $\vert f(gx)-f(x)\vert<\varepsilon/2$ for all $g\in U_{x}$. For every $U_{x}$ $(x\in G)$, pick a symmetric open neighbourhood $V_{x}$ of $e\in G$ such that $V_{x}^{2}\subseteq U_{x}$ using Lemma \ref{lem:prod_nbhd}. Since $\supp f$ is compact, finitely many of the sets $V_{x}x$ $(x\in\supp f)$ cover $\supp f$, say $(V_{x_{k}}x_{k})_{k=1}^{n}$. Define $V=\bigcap_{k=1}^{n}V_{k}$. Then for all $x\in\supp f$ and for all $g\in V$ we have
\begin{displaymath}
  \vert f(gx)-f(x)\vert\le\vert f(gx)-f(x_{k})\vert+\vert f(x_{k})-f(x)\vert<\frac{\varepsilon}{2}+\frac{\varepsilon}{2}=\varepsilon
\end{displaymath}
where $k\in\{1,\ldots, n\}$ is chosen such that $x\in V_{x_{k}}x_{k}$. If $x\notin\supp f$ then for every $g\in V$ either $gx\notin\supp f$, in which case the above inequality is trivial, or $gx\in\supp f$ and we set $y:=gx$. Then $\vert f(gx)-f(x)\vert=\vert f(g^{-1}y)-f(y)\vert$ where $y\in\supp f$ and $g^{-1}\in V$; we may then argue as before.
\end{proof}

Finally, the following facts are useful in various places.

\begin{proposition}\label{prop:topgrp_mult}
Let $G$ be a topological group and $A,B\subseteq G$. If $A$ and $B$ are compact, then $AB$ is compact. If either $A$ or $B$ is open, then $AB$ is open.
\end{proposition}

\begin{proof}
If $A$ and $B$ are compact, then so is $AB$ as the image of the compact set $(A,B)$ under the continuous multiplication map $G\times G\to G$. If either $A$ or $B$ is open, then $AB$ is open as a union of open sets since
$\bigcup_{a\in A}aB=AB=\bigcup_{b\in B}Ab$.
\end{proof}

\begin{proposition}\label{prop:quot_cpct}
Let $G$ be a locally compact Hausdorff group and let $H$ be a subgroup of $G$. Further, let $C\subseteq G/H$ be compact. Then there exists a compact set $K\subseteq G$ such that $\pi(K)\supseteq C$.
\end{proposition}

\begin{proof}
We may cover $G$ by relatively compact open sets $U_{i}$ $(i\in I)$. Since $\pi$ is open and $C\subseteq G/H$ is compact, finitely many of the $\pi(U_{i})$ $(i\in I)$ cover $C$, say $(\pi(U_{k}))_{k=1}^{n}$. Then $K=\bigcup_{k=1}^{n}\overline{U_{k}}$ serves.
\end{proof}

\subsection{Measure Theory}

We now review some basic measure theory in order to give the definition of a Haar measure and some first properties.

\vspace{0.2cm}
Let $X$ be a non-empty set. A \emph{$\sigma$-algebra on $X$} is a set $\calM\subseteq\calP(X)$ of subsets of $X$ which contains the empty set and is closed under taking both complements and countable unions. A pair $(X,\calM)$ where $X$ is a set and $\calM$ is a $\sigma$-algebra on $X$ is a \emph{measurable space}; the sets $E\in\calM$ are \emph{measurable}. Given two measurable spaces $(X,\calM)$ and $(Y,\calN)$, a map $f:X\to Y$ is \emph{measurable} if $f^{-1}(F)\in\calM$ for all $F\in\calN$. For example, let $X$ and $Y$ be topological spaces equipped with their \emph{Borel $\sigma$-algebras} $\calB(X)$ and $\calB(Y)$ respectively, i.e. the $\sigma$-algebra generated by the open sets. Then any continuous map from $X$ to $Y$ is measurable. \emph{We shall always equip topological spaces with their Borel $\sigma$-algebra}.

A \emph{measure} on a measurable space $(X,\calM)$ is a map $\mu:\calM\to\bbR_{\ge 0}\cup\{\infty\}$ which satisfies $\mu(\emptyset)=0$ and is countably additive: whenever $(E_{n})_{n\in\bbN}$ is a sequence of pairwise disjoint measurable sets then $\mu(\bigcup_{n\in\bbN}E_{n})=\sum_{n=1}^{\infty}\mu(E_{n})$. A triple $(X,\calM,\mu)$ where $(X,\calM)$ is a measurable space and $\mu$ is a measure on $(X,\calM)$ is a \emph{measure space}. A set of measure zero is a \emph{null set} and its complement~\emph{conull}.

If $(X,\calM,\mu)$ is a measure space, $(Y,\calN)$ a measurable space and $\varphi:X\to Y$ a measurable map, then $\varphi_{\ast}\mu:\calN\to\bbR_{\ge 0}\cup\{\infty\},\ F\mapsto \mu(\varphi^{-1}(F))$ is the \emph{push-forward measure on $(Y,\calN)$ under $\varphi$}.

The category of measure spaces is designed to allow for the following notion of integration of certain measurable, complex-valued functions on $(X,\calM,\mu)$.

\vspace{0.2cm}
\begin{enumerate}[(1)]
	\item When $\chi_{E}$ is the characteristic function of a measurable set $E\in\calM$, define
		\begin{displaymath}
			\int_{X}\chi_{E}(x)\ \mu(x)=\mu(E).
		\end{displaymath}
	\item When $f=\sum_{i=1}^{n}\lambda_{i}\chi_{E_{i}}$ is a positive, real linear combination of characteristic functions of measurable sets, a \emph{simple function}, define
		\begin{displaymath}
			\int_{X}f(x)\ \mu(x)=\sum_{i=1}^{n}\lambda_{i}\int_{X}\chi_{E_{i}}(x)\ \mu(x).
		\end{displaymath}
	\item When $f:X\to\bbR$ is measurable and non-negative, define
		\begin{displaymath}
			\int_{X}f(x)\ \mu(x)=\sup_{\varphi}\int_{X}\varphi(x)\ \mu(x)
		\end{displaymath}
		where $\varphi$ ranges over all real-valued simple functions on $X$ with $0\le\varphi\le f$.
	\item When $f:X\to\bbR$ is measurable, decompose
		\begin{displaymath}
			f=f_{+}-f_{-} \quad\text{where}\quad f_{\pm}(x)=\max(\pm f(x),0).
		\end{displaymath}
		When $\int_{X}|f(x)|\ \mu(x)<\infty$, define
		\begin{displaymath}
			\int_{X}f(x)\ \mu(x)=\int_{X}f_{+}(x)\ \mu(x)-\int_{X}f_{-}(x)\ \mu(x).
		\end{displaymath}
	\item When $f:X\to\bbC$ is measurable and \emph{integrable}, i.e. $\int_{X}|f(x)|\ \mu(x)<\infty$, define
		\begin{displaymath}
			\int_{X}f(x)\ \mu(x)=\int_{X}\text{Re}(f(x))\ \mu(x)+i\int_{X}\text{Im}(f(x))\ \mu(x).
		\end{displaymath}
\end{enumerate}
\vspace{0.2cm}

The vector space of equivalence classes of measurable, integrable complex-valued functions on~$X$ modulo equality on a conull set is denoted by $L^{1}(X,\mu)$. Integration constitutes a linear map from $L^{1}(X,\mu)$ to $\bbC$. There is the following change of variables formula.

\begin{proposition}[Change of variables]\label{prop:chan_var}
Let $(X,\calM,\mu)$ be a measure space, $(Y,\calN)$ a measurable space and $\varphi:X\to Y$ a measurable. For every measurable function $f:Y\to\bbC$ and $F\in\calN$ we have
\begin{displaymath}
\int_{F}f(y)\ \varphi_{\ast}\mu(y)=\int_{\varphi^{-1}(F)}f(\varphi(x))\ \mu(x).
\end{displaymath}
whenever either of the two expressions is defined.
\end{proposition}

Next, we recall Fubini's Theorem which reduces integrating over a product space to integrating over the factors. Let $(X,\calM,\mu)$ and $(Y,\calN,\nu)$ be measure spaces. Then so is $(X\times Y,\calM\times\calN,\mu\times\nu)$ where $(\mu\times\nu)$ is defined by $(\mu\times\nu)(E,F):=\mu(E)\nu(F)$ for all $(E,F)\in\calM\times\calN$. Recall that $(X,\calM,\mu)$ is \emph{$\sigma$-finite} if $X$ is a countable union of sets of finite measure.

\begin{theorem}[Fubini]\label{thm:fubini}
Let $(X,\calM,\mu)$ and $(Y,\calN,\nu)$ be $\sigma$-finite measure spaces. Further, let $f:X\times Y\to\bbC$ be measurable with $\int_{X}\int_{Y}\vert f(x,y)\vert\ \nu(y)\ \mu(x)<\infty$. Then $f\in L^{1}(X\times Y,\mu\times\nu)$ and
\begin{displaymath}
  \int_{X}\int_{Y}f(x,y)\ \nu(y)\ \mu(x)=\int_{X\times Y}f(x,y)\ (\mu\times\nu)(x,y)=\int_{Y}\int_{X}f(x,y)\ \mu(x)\ \nu(y).
\end{displaymath}
\end{theorem}

Measures on topological spaces which appear in practice often satisfy the following additional regularity properties.

\begin{definition}[Radon measure]\label{def:radon_measure}
A \emph{Radon measure} on a topological space $X$ is a measure on $(X,\calB(X))$ which satisfies the following properties:
\begin{enumerate}
	\item[\namedlabel{item:LF}{(LF)}] If $K\subseteq X$ is compact, then $\mu(K)<\infty$. \hfill (locally finite)
	\item[\namedlabel{item:OR}{(OR)}] If $E\subseteq X$ is measurable, then $\mu(E)=\inf\{\mu(U)\mid U\supseteq E, U \text{ open}\}$. \hfill (outer regular)
	\item[\namedlabel{item:IR}{(IR)}] If $U\subseteq X$ is open, then $\mu(U)=\sup\{\mu(K)\mid K\subseteq U, K \text{ compact}\}$. \hfill (inner regular)
\end{enumerate}
\end{definition}

The importance of Radon measures is also due to the following result of Riesz which is often employed to define a measure on a given space in the first place.

\begin{theorem}[Riesz]\label{thm:riesz}
Let $X$ be a locally compact Hausdorff space. Further, let $\lambda:C_{c}(X)\to\bbC$ be a positive, i.e. $\lambda(f)\in[0,\infty)$ whenever $f(x)\in[0,\infty)$ for all $x\in X$, linear functional. Then there exists a unique Radon measure $\mu$ on $X$ such that
\begin{displaymath}
	\lambda(f)=\int_{X}f(x)\ \mu(x) \quad\text{for all}\quad f\in C_{c}(X).
\end{displaymath}
Furthermore, $\mu$ satisfies $\mu(U)=\sup\{\lambda(f)\mid f\prec U\}$ and $\mu(K)=\inf\{T(f)\mid K\prec f\}$ for every open set $U\subseteq X$ and every compact set $K\subseteq X$ respectively.
\end{theorem}

\section{Definition}

In the context of topological groups it is natural to look for measures which are invariant under translation. Such measures always exist for locally compact Hausdorff groups.

\begin{definition}[Haar measure]
Let $G$ be a locally compact Hausdorff group. A \emph{left (right) Haar measure on $G$} is a Radon measure $\mu$ on $(G,\calB(G))$ which is non-zero on non-empty open sets and invariant under left-translation (right-translation):
\begin{enumerate}
	\item[\namedlabel{item:NT}{(NT)}] If $U\subseteq X$ is open and non-empty, then $\mu(U)\gneq 0$. \hfill (non-trivial)
	\item[\namedlabel{item:TI}{(TI)}] For all $E\in\calB(G)$ and $g\in G$: $\mu(gE)=\mu(E)$ ($\mu(Eg)=\mu(E)$). \hfill (translation-invariant)
\end{enumerate}
\end{definition}

\begin{theorem}[Haar, Weil]\label{thm:haar}
Let $G$ be a locally compact Hausdorff group. Then there exists a left (right) Haar measure on $G$ which is unique up to strictly positive scalar multiples.
\end{theorem}

We do not prove this theorem here but make the following remark.

\begin{remark}
Whereas the uniqueness statement of Theorem \ref{thm:haar} is not too hard to establish, the existence proof is more involved and not particularly fruitful. For both, see e.g. \cite{Wei65}. However, there are several classes of locally compact Hausdorff groups for which the existence of a Haar measure may be established by more classical means, see Remark \ref{rem:haar_exist}.
\end{remark}

\begin{example}
Let $G$ be a discrete group. Then $\calB(G)=\calP(G)$ and the counting measure on $G$, defined by $\mu:\calP(G)\to\bbR_{\ge 0}\cup\{\infty\}$, $E\mapsto |E|$ is a left and right Haar measure.
\end{example}

More examples of Haar measures are given in Example \ref{ex:haar}. For now, consider the following alternative description of Haar measures: Due to Riesz' Theorem \ref{thm:riesz}, there is a one-to-one correspondence between Haar measures and \emph{Haar functionals}, to be defined shortly, on a given group which is often used to define a Haar measure. Recall that a topological group $G$ acts on $C_{c}(G)$ via the left-regular and the right-regular representations $\lambda_{G}(g)f(x)=f(g^{-1}x)$ and $\varrho_{G}(g)f(x)=f(xg)$ respectively, where $g,x\in G$ and $f\in C_{c}(G)$.

\begin{definition}
Let $G$ be a locally compact Hausdorff group. A \emph{left (right) Haar functional} on $G$ is a non-trivial positive linear functional on $C_{c}(G)$ which is invariant under $\lambda_{G}$ $(\varrho_{G})$.
\end{definition}

\begin{proposition}\label{prop:meas_func_corr}
Let $G$ be a locally compact Hausdorff group. Then there are the following mutually inverse maps.
\begin{displaymath}
 \xymatrix@=1.66cm
 {
      \Phi:\{\text{Haar measures on $G$}\} \ar@^{->}@<.3ex>[r]^-{\text{Integration}} & \{\text{Haar functionals on $G$}\}:\Psi \ar@<.3ex>@^{->}[l]^-{\text{Riesz}}
 }
\end{displaymath}
\end{proposition}

\begin{proof}
The map $\Phi$ is readily checked to range in the positive linear functionals on $C_{c}(G)$. For $\lambda_{G}$-invariance ($\varrho_{G}$-invariance), use the change of variables formula given by Proposition~\ref{prop:chan_var}. As to non-triviality, let $\mu$ be a left (right) Haar measure on $G$ and let $K$ be a compact neighbourhood of some point in $G$. Then $\mu(K)\in(0,\infty)$ by \ref{item:LF} and \ref{item:NT}, and by Urysohn's Lemma \ref{lem:urysohn} there is $f\in C_{c}(G)$ such that $K\prec f\prec G$ and therefore $\Phi\mu(f)=\int_{G}f(g)\ \mu(g)\ge\mu(K)\gneq 0$.

Conversely, if $\lambda$ is a left (right) Haar functional on $G$, its non-triviality translates to \ref{item:NT} for $\mu:=\Psi\lambda$ and its invariance under $\lambda_{G}$ $(\varrho_{G})$ translates to \ref{item:TI} for $\mu$:

Suppose $U$ is a non-empty open set of measure zero with respect to $\mu$. Then any compact set admits a finite cover by left (right) translates of $U$ and hence has measure zero as well. Thus $\lambda(f)=\!\int_{G}f(g)\ \mu(g)=\!\int_{\supp f}f(g)\ \mu(g)=0$ for all $f\in C_{c}(G)$, contradicting the non-triviality of $\lambda$.

As for invariance, suppose that $\lambda$ is $\lambda_{G}$-invariant ($\varrho_{G}$-invariance being handled analogously) and let $E\in\calB(G)$ and $g\in G$. Then \ref{item:OR} implies
\begin{displaymath}
  \mu(gE)=\inf\{\mu(U)\mid U\supseteq gE,\ U \text{ open}\}=\inf\{\mu(gU)\mid U\supseteq E,\ U \text{ open}\}.
\end{displaymath}
Furthermore, by Theorem \ref{thm:riesz} and the $\lambda_{G}$-invariance of $\lambda$ we have
\begin{displaymath}
  \mu(gU)=\sup\{\lambda(f)\mid f\prec gU\}=\sup\{\lambda(\lambda_{G}(g)f)\mid f\prec U\}=\mu(U).
\end{displaymath}
Hence $\mu$ is left-invariant. The assertions $\Psi\circ\Phi=\id$ and $\Phi\circ\Psi=\id$ are immediate.
\end{proof}

\begin{example}\label{ex:haar}
Using Proposition~\ref{prop:meas_func_corr} we now provide further examples of Haar measures.
\begin{enumerate}[(i)]
 \item On $G=(\bbR,+)$, a left- and right Haar measure is given by the Lebesgue measure $\lambda$ which can be defined as the Radon measure associated to the Riemann integral $\int_{\bbR}:C_{c}(\bbR)\to \bbC$.
 \item On $G=(\bbR^{n},+)$, $n\ge 1$, a left- and right Haar measure is given by the $n$-th power of the Lebesgue measure $\lambda$.
 \item On $G=(\bbR^{\ast},\cdot)$, the Lebesgue measure is not translation-invariant. However, the map 
 \begin{displaymath}
  \displaybump \mu:C_{c}(G)\to\bbC,\ f\mapsto\int_{\bbR}f(x)\ \frac{\lambda(x)}{|x|}
 \end{displaymath}
  can be checked to be a left- and right Haar functional using the classical substitution rule. Note that the above integral is always finite as the integrand has compact support. Hence $\mu$ defines a left- and right Haar measure on $G$.
 \item On $G=\GL(n,\bbR)$, $n\ge 1$, the map
 \begin{displaymath}
  \displaybump \mu:C_{c}(G)\to\bbC,\ f\mapsto\int_{G}f(X)\ \frac{\lambda(X)}{|\det X|^{n}}
 \end{displaymath}
 defines a left- and right Haar functional. Here, $\lambda(X):=\prod_{i,j=1}^{n}\lambda(x_{ij})$, where $X=(x_{ij})_{i,j}$, is the Lebesgue measure on $\bbR^{n\cdot n}$ of which $\GL(n,\bbR)$ is an open subset. Again, the integral is finite by compactness of the support of the integrand and invariance is checked by changing variables. Note that the case $G=(\bbR^{\ast},\cdot)$ is contained via $n=1$ in this example.
 
 The fact that $\GL(n,\bbR)$ is an open subset of $\bbR^{n\cdot n}$ is key: The above construction does not work for e.g. $\SL(n,\bbR)$ which is a submanifold of $\bbR^{n\cdot n}$ of strictly smaller dimension. A left- and right Haar measure for $\SL(2,\bbR)$ will be constructed in Example \ref{ex:haar_sl2}.
 %\item On $G=S^{1}$, ...
\end{enumerate}
\end{example}

\begin{remark}\label{rem:haar_exist}
With the correspondence between Haar functionals and Haar measures at hand, we now outline existence proofs of Theorem \ref{thm:haar} for compact Hausdorff groups, Lie groups and totally disconnected locally compact separable Hausdorff groups.
\begin{enumerate}[(i)]
 \item\label{item:haar_compact_groups} \emph{Compact Hausdorff groups}. Let $G$ be a compact Hausdorff group. Then $G$ acts continuously on $C(G)=C_{c}(G)$, equipped with the supremum norm, via the left-regular representation. Therefore, $G$ also acts on the dual space $C(G)^{\ast}$ of $C(G)$ via the adjoint representation $\lambda_{G}^{\ast}$ of~$\lambda_{G}$, which is defined by the relation
  \begin{displaymath}
	\displaybump \langle \lambda_{G}^{\ast}(g)\mu,f\rangle=\langle\mu,\lambda_{G}(g^{-1})f\rangle
  \end{displaymath}
    for all $\mu\in C(G)^{\ast}$ and $f\in C(G)$. Since the set $P(G)$ of probability measures on $G$ is a $\text{weak}^{\ast}$-compact, convex and $\lambda_{G}^{\ast}$-invariant subset of $C(G)^{\ast}$, the compact version of the Kakutani-Markov Fixed Point Theorem (e.g. \cite[Thm. 2.23]{Zim90}) provides a $\lambda_{G}^{\ast}$-fixed point within $P(G)$, i.e. a left-invariant probability measure, which turns out to be a Haar~measure.
  
  \item\label{item:haar_lie_groups} \emph{Lie groups}. Let $G$ be a Lie group with Lie algebra $\text{Lie}(G)\cong\Gamma(\mathrm{T}G)^{G}$, the space of left-invariant vector fields on $G$, which is isomorphic to the tangent space $\mathrm{T}_{e}G$ as a vector space. Further, let $X_{1},\ldots,X_{n}$ be a basis of $\mathrm{T}_{e}G$ with associated left-invariant vector fields $X_{1}^{G},\ldots,X_{n}^{G}\in\Gamma(\mathrm{T}G)^{G}$. Then for each $p\in G$, the tuple $((X_{1}^{G})_{p},\ldots,(X_{n}^{G})_{p})$ is a basis of $\mathrm{T}_{p}G$. For each $i\in\{1,\ldots,n\}$ we may thus define a $1$-form $\omega_{i}$ on $G$ by $(\omega_{i})_{p}((X_{j})_{p})=\delta_{ij}$; in other words, for every $p\in G$ the tuple $((\omega_{1})_{p},\ldots,(\omega_{n})_{p})$ is the basis of $\mathrm{T}_{p}^{\ast}G$ dual to $((X_{1}^{G})_{p},\ldots,(X_{n}^{G})_{p})$. It is readily checked that the left-invariance of $X_{1}^{G},\ldots,X_{n}^{G}$ implies left-invariance of the $\omega_{i}$ $(i\in\{1,\ldots,n\})$ in the sense that $L_{g}^{\ast}\omega_{i}=\omega_{i}$ for all $g\in G$ and $i\in\{1,\ldots,n\}$. As a consequence, the $n$-form $\omega:=\omega_{1}\wedge\cdots\wedge\omega_{n}$ is left-invariant as well since $\wedge$ commutes with pullback. One checks that $\omega$ is nowhere vanishing. Finally, we may orient $G$ so that $\omega$ is positive and hence gives rise to the left Haar functional
  \begin{displaymath}
   \displaybump \lambda_{\omega}:C_{c}(G)\to\bbC,\ f\mapsto\int_{G}f\ \omega
  \end{displaymath}
  which in turn provides a left Haar measure on $G$. See \cite[VIII.2]{Kna02} for details.
  
  \item \emph{Totally disconnected locally compact separable Hausdorff groups}. Let $G$ be of this type. By van Dantzig's theorem \cite{vDa31}, $G$ contains a compact open subgroup $K$. Assuming $G$ to be non-compact, by separability and openness of $K$ there are $g_{n}\in G$ $(n\in\bbN)$ such that $G=\bigsqcup_{n\in\bbN}g_{n}K$. Using part (i), let $\nu$ be a Haar measure on $K$ and let $\nu_{n}:=g_{n\ast}\nu$ be the corresponding measure on $g_{n}K$. Finally, for $E\in\calB(G)$ define
  \begin{displaymath}
   \displaybump \mu(E):=\sum_{n\in\bbN}\nu_{n}(E\cap g_{n}K)=\sum_{n\in\bbN}\nu(g_{n}^{-1}E\cap K)
  \end{displaymath}
  if the sum exists and infinity otherwise. Then $\mu$ is a Radon measure on $G$ which is non-zero on non-empty open sets since $\nu$ is. Also, $\mu$ is left-invariant: Given $g\in G$, there is $\sigma\in S_{\bbN}$ such that $gg_{n}K=g_{\sigma(n)}K$. Then
  \begin{align*}
   \displaybump \mu(g^{-1}E)&=\sum_{n\in\bbN}\nu(g_{n}^{-1}g^{-1}E\cap K)=\sum_{n\in\bbN}\nu(g_{\sigma(n)}^{-1}gg_{n}g_{n}^{-1}g^{-1}E\cap K) \\
   &=\sum_{n\in\bbN}\nu(g_{\sigma(n)}^{-1}E\cap K)=\sum_{n\in\bbN}\nu(g_{n}^{-1}E\cap K)=\mu(E).
  \end{align*}
  where the second equality uses $K$-invariance of $\nu$.
\end{enumerate}
\end{remark}

By Remark \ref{rem:haar_exist}\ref{item:haar_compact_groups}, compact Hausdorff groups have finite Haar measure. The converse also holds.

\begin{proposition}\label{prop:haar_cpct}
Let $G$ be a locally compact Hausdorff group and let $\mu$ be a left (right) Haar measure on $G$. Then $\mu(G)<\infty$ if and only if $G$ is compact.
\end{proposition}

\begin{proof}
If $G$ is compact, then $\mu(G)<\infty$ by Definition \ref{item:LF}. Conversely, suppose that $G$ is not compact and let $U$ be a relatively compact neighbourhood of $e\in G$. Then there is an infinite sequence $(g_{n})_{n\in\bbN}$ of elements of $G$ such that $g_{n}\notin\bigcup_{k<n}g_{k}U$; otherwise $G$ would be compact as a finite union of compact sets. Let $V$ be as in Lemma \ref{lem:prod_nbhd}. Then the sets $g_{n}V$ $(n\in\bbN)$ are pairwise disjoint by the fact that $VV^{-1}\subseteq U$ and the definition of $(g_{n})_{n\in\bbN}$. Therefore, as $V$ has strictly positive measure, $G$ has infinite measure.
\end{proof}

\section{Unimodularity}

We now address and quantify the question whether left and right Haar measures on a given locally compact Hausdorff group coincide.

\begin{definition}\label{def:unimod}
A locally compact Hausdorff group $G$ is \emph{unimodular} if every left Haar measure on $G$ is also a right Haar measure on $G$ and conversely.
\end{definition}

\begin{remark}\label{rem:unimod}
By Theorem \ref{thm:haar}, it suffices in Definition \ref{def:unimod} to ask for every left Haar measure on $G$ to also be a right Haar measure.
\end{remark}

Proposition \ref{prop:class_unimod} below provides several classes of unimodular groups. For now, let $G$ be a locally compact Hausdorff group and let $\mu$ be a left Haar measure on $G$. Then for every $g\in G$, the map $\mu_{g}:\calB(G)\to\bbR_{\ge 0}\cup\{\infty\},\ E\mapsto\mu(Eg)$ is a left Haar measure on $G$ as well. Hence, by uniqueness, there exists a strictly positive real number $\Delta_{G}(g)$ such that $\mu_{g}=\Delta_{G}(g)\mu$, i.e.
\begin{equation}
  \mu(Eg)=\mu_{g}(E)=\Delta_{G}(g)\mu(E) \quad\text{for all}\quad E\in\calB(G).
  \tag{M}
  \label{eq:M}	
\end{equation}
The function $\Delta_{G}:G\to\bbR_{>0}$ is independent of $\mu$ and called the \emph{modular function of $G$}.

Let $\lambda$ be the left Haar functional associated to $\mu$ by Proposition \ref{prop:meas_func_corr}. Then by the change of variables formula of Proposition~\ref{prop:chan_var} applied to $\varphi=R_{g^{-1}}$, equation (\ref{eq:M}) immediately translates to
\begin{equation}
  \lambda(\varrho_{G}(g^{-1})f)=\Delta_{G}(g)\lambda(f) \quad\text{for all}\quad f\in C_{c}(G).
  \tag{M'}
  \label{eq:M'}
\end{equation}

\begin{proposition}
Let $G$ be a locally compact Hausdorff group. Then the modular function $\Delta_{G}$ is a continuous homomorphism from $G$ to $(\bbR_{>0},\cdot)$.
\end{proposition}

\begin{proof}
Let $\mu$ be a left Haar measure on $G$. The homomorphism property is immediate from (\ref{eq:M}): For all $g,h\in G$ we have
\begin{displaymath}
  \Delta_{G}(gh)\mu=\mu_{gh}=(\mu_{g})_{h}=\Delta_{G}(h)\mu_{g}=\Delta_{G}(h)\Delta_{G}(g)\mu=\Delta_{G}(g)\Delta_{G}(h)\mu.
\end{displaymath}
Evaluating on a set of non-zero finite measure, e.g. a compact neighbourhood of some point, proves that indeed $\Delta_{G}(gh)=\Delta_{G}(g)\Delta_{G}(h)$.

As to continuity, note that it suffices to check continuity at $e\in G$, since $\Delta_{G}$ is a homomorphism. Let $\lambda$ be the left Haar functional associated to $\mu$ by Proposition \ref{prop:meas_func_corr} and let $K$ be a compact neighbourhood of $e\in G$. Using Urysohn's Lemma \ref{lem:urysohn}, choose $\varphi\in C_{c}(G)$ such that $K\prec \varphi\prec G$ and $\psi\in C_{c}(G)$ such that $K\supp\varphi\prec\psi\prec G$ (see Proposition \ref{prop:topgrp_mult}). In particular, $\varphi$ is uniformly continuous on the right by Proposition \ref{prop:unif_cont}. Hence, given $\varepsilon>0$, there is a symmetric open neighbourhood $U\subseteq K$ of $e\in G$ such that $\vert\varphi(xg)-\varphi(x)\vert<\varepsilon$ for all $g\in U$. Then by (\ref{eq:M'}) we have
\begin{displaymath}
  \vert\Delta_{G}(g)-1\vert=\frac{1}{\lambda(\varphi)}\left\vert \Delta_{G}(g)\lambda(\varphi)-\lambda(\varphi)\right\vert\le\frac{1}{\lambda(\varphi)}\lambda(\vert\varrho_{G}(g^{-1})\varphi-\varphi\vert\psi)\le\varepsilon\frac{\lambda(\psi)}{\lambda(\varphi)}
\end{displaymath}
for all $g\in U$. Hence $\Delta_{G}$ is continuous at $e\in G$.
\end{proof}

\begin{remark}
We have noticed that for a locally compact Hausdorff group $G$ with left Haar measure $\mu$ and given $g\in G$, the map $\mu_{g}:\calB(G)\to\bbR_{\ge 0}\cup\{\infty\},\ E\mapsto\mu(Eg)$ is a left Haar measure on $G$ as well. This is an instance of the following more general observation: For every continuous automorphism $\alpha\in\Aut(G)$, the map $\mu_{\alpha}:\calB(G)\to\bbR_{\ge 0}\cup\{\infty\},\ E\mapsto\mu(\alpha(E))$ is a left Haar measure on $G$. In this setting, $\mu_{g}=\mu_{\text{int}(g^{-1})}$ where $\text{int}(g):G\to G,\ x\mapsto gxg^{-1}$ denotes conjugation in $G$ by $g$. One may then introduce the general \emph{modular function} $\text{mod}_{G}:\Aut(G)\to(\bbR_{>0},\cdot)$ which remains to be a homomorphism and when $\Aut(G)$ is equipped with the Braconnier topology, a refinement of the compact-open topology, becomes continuous. See e.g. \cite[12.1.12]{Pal01} for details.
\end{remark}

We obtain the following useful criterion for unimodularity.

\begin{corollary}\label{cor:crit_unimod}
A locally compact Hausdorff group $G$ is unimodular if and only if $\Delta_{G}\equiv 1$.
\end{corollary}

\begin{proof}
If $\Delta_{G}\equiv 1$, then $G$ is unimodular by (\ref{eq:M}) and Remark \ref{rem:unimod}. Conversely, if $G$ is unimodular, let $\mu$ be a Haar measure on $G$ and let $E$ be a compact neighbourhood of some point in $G$. Then $\mu(E)\in(0,\infty)$ and hence $\Delta_{G}\equiv 1$ by (\ref{eq:M}).
\end{proof}

Corollary \ref{cor:crit_unimod} provides us with the following list of classes of unimodular groups. Yet another class will be given in Proposition \ref{prop:lat_unimod}.

\begin{proposition}\label{prop:class_unimod}
Let $G$ be a locally compact Hausdorff group. Then $G$ is unimodular if, in addition, it satisfies one of the following properties: being abelian, compact, discrete, topologically simple, connected semisimple Lie or connected nilpotent Lie.
\end{proposition}

\begin{proof}
When $G$ is abelian then $Eg=gE$ for every subset $E\subseteq G$ and all $g\in G$. Hence left-invariance implies right-invariance.

When $G$ is compact and $\mu$ is a left Haar measure on $G$, then $\mu(G)\in(0,\infty)$ by \ref{item:LF} and \ref{item:NT} and therefore $\Delta_{G}\equiv 1$ by (\ref{eq:M}).

For a discrete group, the left Haar measures are the strictly positive scalar multiples of the counting measure which is also right-invariant.

When $G$ is topologically simple, then $\overline{[G,G]}$, which is a closed normal subgroup of $G$, either equals $\{e\}$ or $G$. In the first case $G$ is abelian and hence unimodular. In the latter case, continuity of $\Delta_{G}$ implies $\smash{\Delta_{G}(G)=\Delta_{G}(\overline{[G,G]})\subseteq\overline{\Delta_{G}([G,G])}=\{1\}}$ and hence $G$ is unimodular.

When $G$ is a Lie group, the modular function $\Delta_{G}:G\to(\bbR_{>0},\cdot)$ is a continuous and hence smooth (\cite[Thm. 3.39]{War83}) homomorphism of Lie groups. It is given by $\Delta_{G}(g)=\vert\det\text{Ad}(g)\vert$, where $\text{Ad}:G\to\Aut(\text{Lie}(G))$ is the adjoint representation of $G$, see e.g. \cite[Prop. 8.27]{Kna02}, which follows in the setting of Remark \ref{rem:haar_exist}\ref{item:haar_lie_groups}. In particular, the derivative $D_{e}\Delta_{G}:\text{Lie}(G)\to\bbR$ is a morphism of Lie algebras. When $\text{Lie}(G)$ is semisimple we obtain
\begin{displaymath}
D_{e}\Delta_{G}(\text{Lie}(G))=D_{e}\Delta_{G}([\text{Lie}(G),\text{Lie}(G)])=[D_{e}\Delta_{G}(\text{Lie}(G)),D_{e}\Delta_{G}(\text{Lie}(G))]=\{0\}
\end{displaymath}
as $(\bbR_{>0},\cdot)$ is abelian. Thus $\Delta_{G}\!\equiv\! 1$ by the Lie correspondence, passing to the universal cover of $G$.

For a connected nilpotent Lie group the exponential map $\text{exp}:\text{Lie}(G)\to G$ is surjective, see e.g. \cite[Thm. 1.127]{Kna02}. So for every $g\in G$ there is some $X\in\text{Lie}(G)$ such that $g=\text{exp}(X)$ and
\begin{displaymath}
  \Delta_{G}(g)=\vert\det\text{Ad}(g)\vert=\vert\det\text{Ad}(\exp X)\vert=\vert\det e^{\text{ad}X}\vert=e^{\tr(\text{ad}X)}=1
\end{displaymath}
where the last equality follows from the fact that $\text{ad}X$ is nilpotent as $\text{Lie}(G)$ is.
\end{proof}

\begin{remark}
It can be shown that $G$ is unimodular if and only if $G/Z(G)$ is unimodular, see e.g. \cite[Proposition 25]{Nac76}. Hence any nilpotent locally compact Hausdorff group is unimodular. Solvable groups, however, need not be unimodular, see Example~\ref{ex:non_unimod}\ref{item:solvable_non_unimod}.
\end{remark}

The following proposition provides a class of totally disconnected locally compact Hausdorff groups that are unimodular. Recall that if $T$ is a locally finite tree then $\Aut(T)$ is a totally disconnected locally compact separable Hausdorff group with the permutation topology. We adopt Serre's graph theory conventions, see \cite{Ser80}.

\begin{proposition}\label{prop:unimod_tree}
Let $T=(V,E)$ be a locally finite tree. If $G\le\Aut(T)$ is closed and locally transitive then $G$ is unimodular.
\end{proposition}

\begin{proof}
Let $\mu$ be a left Haar measure on $G$, see Remark \ref{rem:haar_exist}. Since $G$ is locally transitive there is for every triple $(x,e_{0},e)$ of a vertex $x\in V$ and edges $e_{0},e\in E(x)$ an element $g_{e}\in G_{x}$ such that $g_{e}e_{0}=e$. Then $G_{x}=\bigsqcup_{e\in E(x)}g_{e}G_{e_{0}}$. Since $G_{e}=G_{\overline{e}}$ for all $e\in E$ we conclude that $\mu(G_{e})=\mu(G_{e'})$ for all $e,e'\in E$. Given $g\in G$ we therefore have
\begin{displaymath}
 \mu(G_{e})=\mu(G_{ge})=\mu(gG_{e}g^{-1})=\mu(G_{e}g^{-1})=\Delta_{G}(g^{-1})\mu(G_{e}).
\end{displaymath}
Since $\mu(G_{e})\in(0,\infty)$ as a compact open subgroup of $\Aut(T)$ we conclude that $G$ is unimodular.
\end{proof}

\begin{example}\label{ex:non_unimod}
We now provide two related examples of non-unimodular groups, cf. Remark~\ref{rem:non_unimod}.
\begin{enumerate}[(i)]
 \item\label{item:solvable_non_unimod} Consider the group
\begin{displaymath}
 \displaybump P:=\left\{\left.\begin{pmatrix} x & y \\ & x^{-1} \end{pmatrix}\right| x\in\bbR\backslash\{0\},\ y\in\bbR\right\}\le\SL(2,\bbR).
\end{displaymath}
Then the functionals $\mu,\nu :C_{c}(P)\to\bbC$, given by
\begin{displaymath}
 \displaybump \mu:f\mapsto\int_{\bbR^{2}}f(X)\ \frac{\lambda(x)\lambda(y)}{x^{2}} \quad\text{and}\quad \nu:f\mapsto\int_{\bbR^{2}}f(X)\ \lambda(x)\lambda(y)
\end{displaymath}
are left- and right Haar functionals respectively as can be checked by changing variables. However, $P$ is a closed subgroup of $\SL(2,\bbR)$ which is unimodular as a connected simple Lie group by Proposition \ref{prop:class_unimod}. Remark \ref{rem:non_unimod} sheds some light on the origin of this example.
 \item Let $T_{d}=(V,E)$ be the $d$-regular tree and let $\omega\in\partial T_{d}$ be a boundary point of $T_{d}$. Set $G:=\Aut(T_{d})_{\omega}$, the stabiliser of $\omega$ in $\Aut(T_{d})$. Then $G$ is not unimodular: Let $t\in G$ be a translation of length $1$ towards $\omega$ and let $x\in V$ be on the translation axis of $t$, then
  \begin{displaymath}
   \displaybump \Delta(t)=\frac{\mu(G_{x})}{\mu(G_{tx})}=\frac{\mu(G_{x})}{\mu(G_{x,tx})}\frac{\mu(G_{x,tx})}{\mu(G_{tx})}=\frac{[G_{x}:G_{x,tx}]}{[G_{tx}:G_{x,tx}]}=\frac{|G_{x}tx|}{|G_{tx}x|}=\frac{1}{d-1}.
  \end{displaymath}
\end{enumerate}
\end{example}

Uilising the modular function, we can turn left Haar measures into right Haar measures as in the following Proposition. Let $i:G\to G$ denote the inversion map of $G$.

\begin{proposition}\label{prop:inv_right}
Let $G$ be a locally compact Hausdorff group with left Haar measure $\mu$. Then $\overline{\mu}=i_{\ast}\mu:\calB(G)\to\bbR_{\ge 0}\cup\{\infty\},\ E\mapsto \mu(E^{-1})$ is a right Haar measure on $G$ with associated right Haar functional $\varrho:C_{c}(G)\to\bbC,\ f\mapsto\int_{G}f(x)\Delta_{G}(x^{-1})\ \mu(x)$. If $G$ is unimodular, then $\overline{\mu}=\mu$.
\end{proposition}

\begin{proof}
The map $\overline{\mu}$ is readily checked to be a right Haar measure on $G$. The map $\varrho$ is clearly positive and linear. Its non-triviality follows as in the proof of Proposition \ref{prop:meas_func_corr} using that $\Delta_{G}(g)\gneq 0$ for all $g\in G$. As to $\varrho_{G}$-invariance, changing variables and using $R_{g\ast}\mu=\mu_{g^{-1}}$ yields
\begin{align*}
\varrho(\varrho_{G}&(g)f)=\int_{G}f(xg)\Delta_{G}(x^{-1})\ \mu(x)=\int_{G}f(x)\Delta_{G}(gx^{-1})\ \mu_{g^{-1}}(x)= \\
&=\int_{G}f(x)\Delta_{G}(g)\Delta_{G}(x^{-1})\Delta_{G}(g^{-1})\ \mu(x)=\int_{G}f(x)\Delta_{G}(x^{-1})\ \mu(x)=\varrho(f).
\end{align*}
for every $f\in C_{c}(G)$ and $g\in G$. Overall, $\varrho$ is a right Haar functional on $G$.

Now, let $\Phi\overline{\mu}$ denote the right Haar functional associated to $\overline{\mu}$ by Proposition \ref{prop:meas_func_corr}. Then there is a strictly positive real number $c$ such that $\Phi\overline{\mu}=c\varrho$. Applying the change of variables formula given by Proposition~\ref{prop:chan_var}, we obtain for all $f\in C_{c}(G)$:
\begin{align*}
  \int_{G}f(x)\ \overline{\mu}(x)&=c\int_{G}f(x)\Delta_{G}(x^{-1})\ \mu(x)=c\int_{G}f(x^{-1})\Delta_{G}(x)\ \overline{\mu}(x) \\
  &=c^{2}\int_{G}f(x^{-1})\Delta_{G}(x)\Delta_{G}(x^{-1})\ \mu(x)=c^{2}\int_{G}f(x)\ \overline{\mu}(x).
\end{align*}
Let $K$ be a compact symmetric neighbourhood of a point in $G$ and $f\in C_{c}(G)$ with $K\prec f\prec G$. Then $\int_{G}f(x^{-1})\ \mu(x)\in(0,\infty)$ and hence $c=1$. Henceu unimodularity of $G$ implies $\mu=\overline{\mu}$.
\end{proof}

\section{Coset spaces}\label{sec:coset_spaces}

Let $G$ be a locally compact Hausdorff group and let $H$ be a closed subgroup of $G$. When $H$ is normal in $G$, there exists a left (right) Haar measure on $G/H$ by Theorem \ref{thm:haar}. We now address the question under which circumstances there exists a $G$-invariant Radon measure on $G/H$ that is non-zero on non-empty open sets when $H$ is not normal in $G$. We shall refer to such a measure as a \emph{Haar measure on $G/H$} by abuse of notation. The following example shows that a Haar measure on $G/H$ may or may not exist.

\begin{example}
Let $G=\SL(2,\bbR)$.
\begin{enumerate}[(i)]
 \item Consider the natural, transitive action of $G$ on $X=\bbR^{2}\backslash\{0\}$ and the stabiliser
   \begin{displaymath}
    \displaybump H:=\stab_{G}((1,0)^{T})=\left\{\left.\begin{pmatrix}1 & x \\ & 1\end{pmatrix}\right| x\in\bbR\right\}.
   \end{displaymath}
   Then $G/H\cong X$ on which the restricted $2$-dimensional Lebesgue measure is a Haar measure.
 \item Consider the natural, transitive action of $G$ on $X=\bbP^{1}\bbR=\{V\le\bbR^{2}\mid \dim V=1\}$ and
   \begin{displaymath}
   \displaybump H:=\stab_{G}(\langle e_{1}\rangle)=\left\{\left.\begin{pmatrix}x & y \\ & x^{-1}\end{pmatrix}\right| x\in\bbR\backslash\{0\},y\in\bbR\right\}.
   \end{displaymath}
   Here, $G/H\cong X$ does not admit a Haar measure: Indeed, consider the compact subsets $E_{1}:=\{\langle(1,t)^{T}\rangle\mid t\in[0,1]\}$ and $E_{2}:=\{\langle(t,1)^{T}\rangle\mid t\in[0,1]\}$ of $\bbP^{1}\bbR$. Then
   \begin{displaymath}
    \displaybump \begin{pmatrix} 1 & -1 \\ & 1 \end{pmatrix}E_{1}=E_{1}\cup E_{2} \quad\text{and}\quad \begin{pmatrix} 1 & -1 \\ 1 & \end{pmatrix}E_{1}=E_{2}.
   \end{displaymath}
   A Haar measure on $G/H$ would assign finite non-zero measure to the compact sets $E_{1}$ and $E_{2}$. Combined with $G$-invariance contradicts the above two equalities. Note that $H$ is the non-unimodular group of Example \ref{ex:non_unimod}.
\end{enumerate}
\end{example}

\begin{theorem}\label{thm:measure_coset}
Let $G$ be a locally compact Hausdorff group with left Haar measure $\mu$ and let $H$ be a closed subgroup of $G$ with left Haar measure $\nu$. Then there exists a Haar measure $\xi$ on $G/H$ if and only if $\Delta_{G}|_{H}\equiv\Delta_{H}$. In this case, $\xi$ is unique up to strictly positive scalar multiples and suitably normalized satisfies for all $f\in C_{c}(G)$:
\begin{equation*}
  \int_{G}f(g)\ \mu(g)=\int_{G/H}\int_{H}f(gh)\ \nu(h)\ \xi(gH).
  \tag{W}
  \label{eq:W}
\end{equation*}
\end{theorem}

In the context of Theorem \ref{thm:measure_coset}, formula \eqref{eq:W} can be extended to hold for $f\in L^{1}(G)$, see e.g. \cite[Theorem 7.12]{KL06} and the explanations around it.

\begin{proof}(Theorem \ref{thm:measure_coset}, ``$\Rightarrow$'').
If $\xi$ exists as above, then the map
\begin{displaymath}
  \lambda:C_{c}(G)\to\bbC,\ f\mapsto \int_{G/H}\int_{H}f(gh)\ \nu(h)\ \xi(gH)
\end{displaymath}
is a left Haar functional on $G$ and therefore defines a left Haar measure $\mu$ on $G$. In particular, $\lambda(\varrho_{G}(t^{-1})f)=\Delta_{G}(t)\lambda(f)$ for all $t\in G$ and $f\in C_{c}(G)$ by (\ref{eq:M'}). On the other hand, we have for all $t\in H$ and $f\in C_{c}(G)$:
\begin{align*}
 \lambda(\varrho_{G}(t^{-1})f)&=\int_{G/H}\int_{H}(\varrho_{G}(t^{-1})f)(gh)\ \nu(h)\ \xi(gH) \\
 &=\int_{G/H}\int_{H}\Delta_{H}(t)f(gh)\ \nu(h)\ \xi(gH)=\Delta_{H}(t)\lambda(f).
\end{align*}
Using Urysohn's Lemma \ref{lem:urysohn}, choose $f\in C_{c}(G)$ to satisfy $K\prec f\prec G$ where $K$ is a compact neighbourhood of some point in $G$. Then $\int_{G}f(g)\ \mu(g)=\lambda(f)\in(0,\infty)$ and hence $\Delta_{G}|_{H}\equiv\Delta_{H}$.
\end{proof}

The proof of the converse assertion of Theorem \ref{thm:measure_coset} relies on the following description of compactly supported functions on $G/H$. Once more, Riesz' Theorem \ref{thm:riesz} is used to produce a measure.

\begin{lemma}\label{lem:coset}
Let $G$ be a locally compact Hausdorff group and $H$ a closed subgroup of $G$ with left Haar measure $\nu$. Then the following map is surjective:
\begin{displaymath}
C_{c}(G)\to C_{c}(G/H),\ f\mapsto \left(f_{H}:gH\mapsto \int_{H}f(gh)\ \nu(h)\right).
\end{displaymath}
\end{lemma}

\begin{proof}
Several things need to be checked. First of all, for all $f\in C_{c}(G)$ and for all $gH\in G/H$, the integral $\int_{H}f(gh)\ \nu(h)$ is independent of the representative of $gH$ and finite. Next, for all $f\in C_{c}(G)$, the function $f_{H}$ is continuous as a parametrized integral as in the proof of the continuity of the modular function. Clearly, $\supp f_{H}\subseteq \pi(\supp(f))$ and hence $f_{H}\in C_{c}(G/H)$. It remains to prove surjectivity. To this end, let $F\in C_{c}(G/H)$. Pick $K\subseteq G$ such that $\pi(K)\supseteq\supp F$ (Proposition~\ref{prop:quot_cpct}) and let $\eta\in C_{c}(G)$ satisfying $K\prec\eta$ (Urysohn's Lemma~\ref{lem:urysohn}). Now define $f\in C_{c}(G)$ by
\begin{displaymath}
 f:G\to\bbC,\ g\mapsto\begin{cases} \frac{F(gH)\eta(g)}{\eta_{H}(gH)} & \eta_{H}(gH)\neq 0 \\ 0 & \eta_{H}(gH)= 0\end{cases}
\end{displaymath}
Again, we need to show that this function is continuous and has compact support. As for compact support, clearly $\supp f\subseteq\supp\eta$. In fact, if $G$ was compact, we could choose $\eta\equiv 1$. To show that $f$ is continuous, we show that it is continuous at every point of two open sets $U_{1}\subseteq G$ and $U_{2}\subseteq G$ satisfying $U_{1}\cup U_{2}=G$. On the set $U_{1}:=\{g\in G\mid \eta_{H}(gH)\neq 0\}$ it is continuous as a quotient of continuous functions; and on the set $U_{2}:=G\backslash KH$ it is continuous because it vanishes. Further, if $g\not\in U_{1}$, then $0=\eta_{H}(gH)=\int_{H}\eta(gh)\ \nu(h)$. Since $\eta$ is a non-negative continuous function, this implies $\eta(gh)=0$ for all $h\in H$, hence $g\not\in KH$, i.e. $g\in U_{2}$. With continuity and compact support established, it remains to show that $f_{H}\equiv F$. To this end, we compute
\begin{displaymath}
 f_{H}(gH)=\int_{H}\frac{F(ghH)\eta(gh)}{\eta_{H}(ghH}\ \nu(h)=F(gH)\frac{\int_{H}\eta(gh)\ \nu(h)}{\eta_{H}(gH)}=F(gH).
\end{displaymath}
Hence the map $(-)_{H}:C_{c}(G)\to C_{c}(G/H)$ is surjective.
\end{proof}

\begin{proof}(Theorem \ref{thm:measure_coset}, ``$\Leftarrow$'').
Lemma~\ref{lem:coset} allows us to pick a be a right-inverse $s:C_{c}(G/H)\to C_{c}(G)$ for the map $(-)_{H}:C_{c}(G)\to C_{c}(G/H),\ f\mapsto f_{H}$ of the same lemma. Now consider the map
\begin{displaymath}
  \lambda:C_{c}(G/H)\to\bbC,\ f\mapsto\int_{G}(s f)(g)\ \mu(g).
\end{displaymath}
Once $\lambda$ is independent of $s$, it is a positive linear functional. To prove that it is independent of $s$, it suffices to show that $\int_{G}f(g)\ \mu(g)=0$ whenever $f_{H}\equiv 0$. By Lemma~\ref{lem:coset} and Urysohn's Lemma~\ref{lem:urysohn} there is a function $\eta\in C_{c}(G)$ such that $(\supp f)H\!\prec\eta_{H}\!\prec G/H$. Then by Proposition \ref{prop:inv_right} we have
\begin{align*}
  \int_{G}f(g)\ \mu(g)&=\int_{G}\eta_{H}(gH)f(g)\ \mu(g)=\int_{G}\int_{H}\eta(gh)f(g)\ \nu(h)\ \mu(g) \\
  &=\int_{G}\int_{H}\eta(gh^{-1})f(g)\Delta_{H}(h^{-1})\ \nu(h)\ \mu(g).
\end{align*}
We may as well integrate over the compact spaces $\supp f\subseteq G$ and $(\supp\eta)^{-1}\supp f\cap H\subseteq H$ (Proposition \ref{prop:topgrp_mult}). Fubini's Theorem~\ref{thm:fubini} then allows us to continue the above computation by
\begin{displaymath}
    =\int_{H}\int_{G}\eta(gh^{-1})f(g)\Delta_{H}(h^{-1})\ \mu(g)\ \nu(h)=\int_{H}\int_{G}\eta(g)f(gh)\Delta_{H}(h^{-1})\Delta_{G}(h)\ \mu(g)\ \nu(h).
\end{displaymath}
Applying Fubini's Theorem \ref{thm:fubini} again, we deduce using that $\Delta_{G}|_{H}\equiv\Delta_{H}$ and $f_{H}\equiv 0$:
\begin{displaymath}
  =\int_{G}\eta(g)\int_{H}f(gh)\ \nu(h)\ \mu(g)=\int_{G}\eta(g)f_{H}(gH)=0
\end{displaymath}
which completes the proof that $\lambda$ is a positive linear functional. Hence, by Riesz' Theorem \ref{thm:riesz}, there exists a unique Radon measure $\xi$ on $G/H$ such that 
\begin{align*}
 \int_{G}(s f)(g)\ \mu(g)=\lambda(f)&=\int_{G/H}f(gH)\ \xi(gH) \\
 &=\int_{G/H}(s f)_{H}(gH)\ \xi(gH)=\int_{G/H}\int_{H}(s f)(gh)\ \nu(h)\ \xi(gH)
\end{align*}
for all $f\in C_{c}(G/H)$. The measure $\xi$ is checked to be non-zero on non-empty open sets as well as $G$-invariant, i.e. $\xi$ is a Haar measure on $G/H$. Since the above equation is independent of $s$, we may as well start with a function $f\in C_{c}(G)$, thus proving the existence of a unique Haar measure $\xi$ on $G/H$ satisfying (\ref{eq:W}). To complete the proof, we need to show that any Haar measure on $G/H$ (not necessarily satisfying (\ref{eq:W})) is a strictly positive scalar multiple of $\xi$: Let $\xi_{1},\xi_{2}$ be Haar measures on $G/H$. Then there are left Haar measures $\mu_{1},\mu_{2}$ on $G$ satisfying (\ref{eq:W}) for $\xi_{1}$ and $\xi_{2}$ respectively (see the converse direction of the proof). By uniqueness, $\mu_{2}=c\mu_{1}$ for some strictly positive real number $c$. Then $\xi_{2}$ and $c\xi_{1}$ both satisfy (\ref{eq:W}) for $\mu_{2}$. By the uniqueness proven above, $\xi_{2}=c\xi_{1}$.
\end{proof}

\begin{remark}
Retain the notation of Theorem~\ref{thm:measure_coset}. When $G$ is compact, we may choose $\eta\equiv 1$ in the proof of Lemma~\ref{lem:coset}. The constructed left Haar functional on $G/H$ is then given by
\begin{displaymath}
 \lambda:C_{c}(G/H)\to\bbC,\ f\mapsto \int_{G}\frac{f(gH)}{1_{H}(gH)}\ \mu(g)=\frac{1}{\nu(H)}\int_{G}f(gH)\ \mu(g).
\end{displaymath}
Notice that $\nu(H)$ is finite by Proposition \ref{prop:haar_cpct} given that  $H$ is compact as a closed subset of a compact space. Now, it is a fact (see \cite[Thm. 7.12]{KL06}) that the Haar measure $\xi$ on $G/H$ associated to $\lambda$ can be computed by evaluating $\lambda$ on characteristic functions. Thus, when $E\subseteq G/H$ is measurable,
\begin{displaymath}
	\xi(E)=\frac{\mu(\pi^{-1}(E))}{\nu(H)}, \quad\text{in particular}\quad \xi(G/H)=\frac{\mu(G)}{\nu(H)}.
\end{displaymath}
The auxiliary function $\eta$ merely mends the issues that arise when $G$ is not compact.
\end{remark}

\begin{example}\label{ex:haar_sl2}
To illustrate the usefulness of Theorem~\ref{thm:measure_coset}, we now provide a Haar functional for $G:=\SL(2,\bbR)$. Recall that $G$ acts transitively on the upper half plane $\bbH:=\{z\in\bbC\mid \text{Im}(z)>0\}$ via fractional linear transformations:
\begin{displaymath}
	\begin{pmatrix} a & b \\ c & d \end{pmatrix}z:=\frac{az+b}{cz+d} \quad\text{and}\quad \begin{pmatrix} \sqrt{y} & x\sqrt{y}^{-1} \\ & \sqrt{y}^{-1} \end{pmatrix}i=x+iy
\end{displaymath}
for $x\in\bbR$ and $y\in\bbR_{>0}$. Also, one readily verifies that $H:=\text{stab}_{G}(i)=\SO(2,\bbR)$. Hence the maps
\begin{displaymath}
	G/H\to\bbH,\ gH\mapsto gi \quad\text{and}\quad \bbH\to G/H,\ x+iy\mapsto\begin{pmatrix} \sqrt{y} & x\sqrt{y}^{-1} \\ & \sqrt{y}^{-1} \end{pmatrix}
\end{displaymath}
are mutually inverse $G$-isomorphisms. In fact they are homeomorphisms. Since $G$ is unimodular as a connected semisimple Lie group and $H$ is unimodular as a compact group by Proposition \ref{prop:class_unimod}, we obtain a Haar measure $\xi$ on $G/H\cong\bbH$ by Theorem \ref{thm:measure_coset}. Let $\nu$ be a left Haar measure on $H$. Then
\begin{displaymath}
	C_{c}(G)\to\bbC,\ f\mapsto \int_{G/H}\int_{H}f(gH)\ \nu(h)\ \xi(gH)
\end{displaymath}
is a left Haar functional on $G$. To make this computable, we use the homeomorphisms $H\cong S^{1}$ and $G/H\cong\bbH$ to change variables via Proposition \ref{prop:chan_var}, and the $\SL(2,\bbR)$-invariant Radon measure on $\bbH$ that stems from hyperbolic geometry. All together, the Haar functional on $G=\SL(2,\bbR)$ reads
\begin{displaymath}
	f\mapsto\int_{-\infty}^{\infty}\int_{0}^{\infty}\int_{0}^{2\pi}f\left(\begin{pmatrix}\sqrt{y} & x\sqrt{y}^{-1} \\ & \sqrt{y}^{-1}\end{pmatrix}\begin{pmatrix}\cos\theta & \sin\theta \\ -\sin\theta & \cos\theta\end{pmatrix}\right)\ d\theta\ \frac{d\lambda(y)\ d\lambda(x)}{y^{2}}.
\end{displaymath}
\end{example}

\begin{remark}\label{rem:non_unimod}
In the setting of Example \ref{ex:haar_sl2}, the group $P$ of Example \ref{ex:non_unimod} is the stabiliser in $\SL(2,\bbR)$ of the boundary point of $\bbH$ associated to the (unit-speed) geodesic $\gamma:[0,\infty)\to\bbH,\ t\mapsto i+ie^{it}$. Basically, $P$ translates $\gamma$ to asymptotic geodesics. More general, when $M$ is a symmetric space of non-compact type such as $\SL(n,\bbR)/\SO(n)$, let $G:=\mathrm{Iso}(M)^{\circ}$, $p\!\in\! M$ and $x\in\partial M$ a boundary~point. Then there is a strong dichotomy between $\mathrm{stab}_{G}(p)$ and $\mathrm{stab}_{G}(x)$ that pertains to compactness, connnectedness, transitivity, conjugacy and unimodularity. See \cite[2.17]{Ebe96} for details.
\end{remark}

\begin{comment}
\begin{displaymath}
\begin{tabular}{c|c}
$\mathrm{stab}_{G}(p)$ & $\mathrm{stab}_{G}(x)$ \\ \hline 
compact & non-compact \\
connected & not in general connected  \\
not transitive on $M$ & transitive on $M$ \\
one conjugacy class & in general several conjugacy classes \\ \hline
unimodular & not in general unimodular
\end{tabular}
\end{displaymath}
\end{comment}

\subsection{Discrete Subgroups}

When, in the discussion above, $\Gamma:=H$ is a discrete subgroup of $G$ and $G$ is second-countable, then integration over $G/\Gamma$ can be realized by integrating over a \emph{fundamental domain} for $G/\Gamma$ in $G$. In the following, we pick the counting measure $\nu$ as the Haar measure on $\Gamma$.

\begin{definition}
Let $G$ be a locally compact Hausdorff group and let $\Gamma$ be a discrete subgroup of $G$. A \emph{strict fundamental domain} for $G/\Gamma$ in $G$ is a set $F\in\calB(G)$ such that $\pi:F\to G/\Gamma$ is a bijection. A \emph{fundamental domain} for $G/\Gamma$ in $G$ is a set $F\in\calB(G)$ which differs from a strict fundamental domain by a set of measure zero with respect to any left Haar measure on $G$.
\end{definition}

\begin{proposition}\label{prop:dis_fund}
Let $G$ be a locally compact Hausdorff, second-countable group with a discrete subgroup $\Gamma$. Then there exists a fundamental domain for $G/\Gamma$ in $G$.
\end{proposition}

\begin{remark}
Note that, in Proposition \ref{prop:dis_fund}, second-countability of $G$ implies that $\Gamma$ is countable.
\end{remark}

\begin{proof}(Proposition \ref{prop:dis_fund}).
The canonical projection $\pi:G\to G/\Gamma$ is a local homeomorphism. In view of second-countability, this implies the existence of an open cover $(U_{n})_{n\in\bbN}$ of $G$ such that $\pi:U_{n}\to\pi(U_{n})$ is a homeomorphism for every $n\in\bbN$. Let $F_{1}=U_{1}$ and define inductively $F_{n}=U_{n}\backslash(U_{n}\cap\pi^{-1}\pi(\bigcup_{k<n}U_{k}))$. Then $F:=\bigcup_{n\in\bbN}F_{n}$ is a fundamental domain for $G/\Gamma$ in $G$.
\end{proof}

Integration over $G/\Gamma$ now reduces to integration over $G$ as follows.

\begin{proposition}\label{prop:int_discrete}
Let $G$ be a locally compact Hausdorff, second-countable group with left Haar measure $\mu$ and let $\Gamma$ be a discrete subgroup of $G$. Assume that $\Delta_{G}|_{\Gamma}\equiv\Delta_{\Gamma}$. Further, let $F$ be a fundamental domain for $G/\Gamma$ in $G$. Then a Haar measure $\xi$ on $G/\Gamma$ satisfying (\ref {eq:W}) exists and is associated to the following functional: $\lambda:C_{c}(G/\Gamma)\to\bbC,\ f\mapsto\int_{F}f(g\Gamma)\ \mu(g)$, i.e.
\begin{displaymath}
  \int_{G/\Gamma}f(g\Gamma)\ \xi(g\Gamma)=\int_{F}f(g\Gamma)\ \mu(g) \quad\text{for all}\quad f\in C_{c}(G/\Gamma).
\end{displaymath}
\end{proposition}

\begin{proof}
The functional $\lambda$ is positive and linear. The associated Radon measure $\xi$ on $G/\Gamma$ is checked to be non-zero on non-empty open sets and $G$-invariant. Hence $\xi$ is a Haar measure on $G/\Gamma$. To prove that it satisfies (\ref{eq:W}), note that changing $F$ by a set of measure zero, we may assume that $F$ is a strict fundamental domain. Then $G$ is a countable disjoint union $G=\bigsqcup_{\gamma\in\Gamma}F\gamma$ and hence
\begin{align*}
  \int_{G}&f(g)\ \mu(g)=\sum_{\gamma\in\Gamma}\int_{F\gamma}f(g)\ \mu(g)=\sum_{\gamma\in\Gamma}\int_{F}f(g\gamma)\ \mu(g)=\int_{\Gamma}\int_{F}f(g\gamma)\ \mu(g)\ \nu(\gamma) \\ 
  &=\int_{F}\int_{\Gamma}f(g\gamma)\ \nu(\gamma)\ \mu(g)=\int_{F}f_{\Gamma}(g\Gamma)\ \mu(g)=\int_{G/\Gamma}f_{\Gamma}(g\Gamma)\ \xi(g\Gamma)=\int_{G/\Gamma}\int_{\Gamma}f(g\gamma)\ \nu(\gamma)\ \xi(g\Gamma)
\end{align*}
for all $f\in C_{c}(G)$, where the second equality follows from the assumption that $\Delta_{G}|_{\Gamma}\equiv\Delta_{\Gamma}\equiv 1$, and the application of Fubini's Theorem \ref{thm:fubini} is valid since $G$ is $\sigma$-finite as a locally compact, second-countable space and $\Gamma$ is $\sigma$-finite because it is countable.
\end{proof}

\begin{remark}
Retain the notation of Proposition \ref{prop:int_discrete}. The assumption that $\Delta_{G}|_{\Gamma}\equiv\Delta_{\Gamma}$ is not automatic. For instance, the subgroup
\begin{displaymath}
 \Gamma:=\left\{\left.\begin{pmatrix} e^{t} &  \\ & e^{-t} \end{pmatrix}\right| t\in\bbZ\right\}
\end{displaymath}
of the group $P$ of Example \ref{ex:non_unimod} is isomorphic to $\bbZ$ and discrete in $P$. However, for $\gamma:=\mathrm{diag}(e^{t},e^{-t})$ we have $\Delta_{P}(\gamma)=e^{-2t}\neq 1\equiv\Delta_{\Gamma}$ by Example \ref{ex:non_unimod} whenever $t\neq 0$.
\end{remark}

We end with a result about groups containing lattices. Recall that a \emph{lattice} $\Gamma$ in a locally compact Hausdorff group $G$ is discrete subgroup such that $G/\Gamma$ supports a finite Haar measure.

\begin{proposition}\label{prop:lat_unimod}
A locally compact Hausdorff group $G$ containing a lattice $\Gamma$ is unimodular.
\end{proposition}

\begin{proof}
Suppose $\Gamma$ is a lattice in $G$. Since $G/\Gamma$ supports a finite Haar measure $\xi$, Theorem \ref{thm:measure_coset} implies that $\Delta_{G}|_{\Gamma}\equiv\Delta_{\Gamma}\equiv 1$ and hence $\ker\Delta_{G}\supseteq\Gamma$. Therefore, $\Delta_{G}$ factors through $G\to G/\Gamma$ via $\smash{\widetilde{\Delta}_{G}:G/\Gamma\to(\bbR_{\ge 0}^{\ast},\cdot)}$. Then $\smash{(\widetilde{\Delta}_{G})_{\ast}\xi}$ is a non-zero, finite measure on $\bbR_{\ge 0}^{\ast}$ which is invariant under the image of $\Delta_{G}$. This forces $\Delta_{G}\equiv 1$.
\end{proof}

\begin{comment}
\begin{remark}
The assumption of Proposition \ref{prop:int_discrete} that $\Delta_{G}|_{\Gamma}\equiv\Delta_{\Gamma}$ is automatic in the case of lattices by the proof of Proposition \ref{prop:lat_unimod}. However, it is not so for general discrete subgroups.

For instance, let $T_{d}=(X,Y)$ be the $d$-regular tree $(d\ge 3)$; we adopt Serre's conventions for graph theory, see \cite{Ser80}. Then $\Aut(T_{d})$ is a totally disconnected, locally compact Hausdorff group for the compact-open topology with basis of open sets the pointwise stabilisers of finite subgraphs.
Now, let $H\le\Aut(T_{d})$ be a closed subgroup and fix $x\in X$. We determine the modular function of $H$: Let $\mu$ be the left Haar measure of $H$. Since $H_{x}:=\stab_{H}(x)$ is compact and open, we have $\mu(H)\in(0,\infty)$ and hence we have for every $h\in H$:
\begin{displaymath}
 \Delta_{H}(h)=\frac{\mu(H_{x})}{\mu(H_{x}h)}=\frac{\mu(H_{x})}{\mu(h^{-1}H_{x}h)}=\frac{\mu(H_{x})}{\mu(h^{-1}H_{x}h)}
\end{displaymath}
\end{remark}
\end{comment}

\bibliographystyle{amsalpha}
\bibliography{haar_measures}

\end{document}